\documentclass[a4paper,english,fontsize=10pt,parskip=half,abstracton]{scrartcl}
\usepackage{babel}
\usepackage[utf8]{inputenc}
\usepackage[T1]{fontenc}
\usepackage[a4paper,left=20mm,right=20mm,top=30mm,bottom=30mm]{geometry}
\usepackage{amsmath}
\usepackage{amsthm}
\usepackage{amssymb}
\usepackage{enumerate}
\usepackage[bookmarks=false,
            pdftitle={Brauer's Height Zero Conjecture for metacyclic defect groups},
            pdfauthor={Benjamin Sambale},
            pdfkeywords={blocks, metacyclic defect group},
            pdfstartview={FitH}]{hyperref}

\newtheorem{Theorem}{Theorem}[section]
\newtheorem{Lemma}[Theorem]{Lemma}
\newtheorem{Proposition}[Theorem]{Proposition}
\newtheorem{Corollary}[Theorem]{Corollary}
\numberwithin{equation}{section}

\renewcommand{\phi}{\varphi}
\newcommand{\C}{\operatorname{C}}

\newcommand{\N}{\operatorname{N}}
\newcommand{\Z}{\operatorname{Z}}
\newcommand{\cohom}{\operatorname{H}}
\newcommand{\Aut}{\operatorname{Aut}}
\newcommand{\Out}{\operatorname{Out}}
\newcommand{\pcore}{\operatorname{O}}
\newcommand{\Irr}{\operatorname{Irr}}

\newcommand{\Gal}{\operatorname{Gal}}
\newcommand{\Ob}{\operatorname{Ob}}
\newcommand{\Hom}{\operatorname{Hom}}

\mathchardef\ordinarycolon\mathcode`\:  
 \mathcode`\:=\string"8000
 \begingroup \catcode`\:=\active
   \gdef:{\mathrel{\mathop\ordinarycolon}}
 \endgroup

\title{Brauer's Height Zero Conjecture\\ for metacyclic defect groups}
\author{Benjamin Sambale}
\date{\today}

\begin{document}
\frenchspacing
\maketitle
\begin{abstract}\noindent
We prove that Brauer's Height Zero Conjecture holds for $p$-blocks of finite groups with metacyclic defect groups. 
If the defect group is nonabelian and contains a cyclic maximal subgroup, we obtain the distribution into $p$-conjugate and $p$-rational irreducible characters. Then the Alperin-McKay Conjecture follows provided $p=3$. Along the way we verify a few other conjectures.
Finally we consider the extraspecial defect group of order $p^3$ and exponent $p^2$ for an odd prime more closely. Here for blocks with inertial index $2$ we prove the Galois-Alperin-McKay Conjecture by computing $k_0(B)$. Then for $p\le 11$ also Alperin's Weight Conjecture follows. This improves some results of [Gao, 2012], [Holloway-Koshitani-Kunugi, 2010] and [Hendren, 2005].
\end{abstract}

\textbf{Keywords:} Brauer's Height Zero Conjecture, metacyclic defect groups, Alperin's Weight Conjecture\\
\textbf{AMS classification:} 20C15, 20C20

\section{Introduction}
An important task in representation theory is the determination of the invariants of a block of a finite group when its defect group is given. For a $p$-block $B$ of a finite group $G$ we are interested in the number $k(B)$ of irreducible ordinary characters and the number $l(B)$ of irreducible Brauer characters of $B$. Let $D$ be a defect group of $B$. Then the irreducible ordinary characters split into $k_i(B)$ characters of height $i\ge 0$. Here the \emph{height} $h(\chi)$ of a character $\chi$ in $B$ is defined by $\chi(1)_p=p^{h(\chi)}|G:D|_p$.

If $p=2$, the block invariants for several defect groups were obtained in the last years. In particular the invariants are known if the defect group is metacyclic (see \cite{Sambale}). However, for odd primes $p$ the situation is somewhat more complicated. Here even in the smallest interesting example of an elementary abelian defect group of order $9$ the block invariants are not determined completely (see \cite{Kiyota}). Nevertheless Brauer's $k(B)$-Conjecture and Olsson's Conjecture were proved for all blocks with metacyclic defect groups by \cite{GaoBrauer,YangOlsson}.
Following these lines we obtain in this paper that also Brauer's Height Zero Conjecture is fulfilled for these blocks. 
The proof uses the notion of lower defect groups and inequalities from \cite{HKS}. Moreover, if $G$ is $p$-solvable, we obtain the algebra structure of $B$ with respect to an algebraically closed field of characteristic $p$.
If one restricts to blocks with maximal defect, the precise invariants were determined in \cite{Gaofull}. We can confirm at least some of these values. For principal blocks there is even a perfect isometry between $B$ and its Brauer's correspondent in $\N_G(D)$ by the main theorem of \cite{WatanabePerfIso}.

In the second part of the paper we consider the (unique) nonabelian $p$-group with a cyclic subgroup of index $p$ as a special case. Here the difference $k(B)-l(B)$ is known from \cite{GaoZeng}. We confirm this result and derive the distribution into $p$-conjugate and $p$-rational irreducible characters. We also show that $k_i(B)=0$ for $i\ge 2$. This implies various numerical conjectures.
Moreover, it turns out that the Alperin-McKay Conjecture holds provided $p=3$. This is established by computing $k_0(B)$. Here in case $|D|\le 3^4$ we even obtain the other block invariants $k(B)$, $k_i(B)$ and $l(B)$ which leads to a proof of Alperin's Weight Conjecture in this case.
This generalizes some results from \cite{Holloway}, where these blocks were considered under additional assumptions on $G$.

The smallest nonabelian example for a metacyclic defect group for an odd prime is the extraspecial defect group $p^{1+2}_-$ of order $p^3$ and exponent $p^2$. For this special case Hendren obtained in \cite{Hendren2} some inequalities on the invariants. In \cite{Schulz} one can find results for these blocks under the hypothesis that $G$ is $p$-solvable.
The present paper improves both of these works. In particular if the inertial index $e(B)$ of $B$ is $2$, we verify the Galois-Alperin-McKay Conjecture (see \cite{IsaacsNavarro}), a refinement of the Alperin-McKay Conjecture.
As a consequence, for $p\le 11$ we are able to determine the block invariants $k(B)$, $k_i(B)$ and $l(B)$ completely without any restrictions on $G$. 
Then we use the opportunity to prove severals conjectures including Alperin's Weight Conjecture for this special case. As far as I know, these are the first nontrivial examples of Alperin's Conjecture for a nonabelian defect group for an odd prime.

\section{Brauer's Height Zero Conjecture}
Let $B$ be a $p$-block of a finite group $G$ with metacyclic defect group $D$. Since for $p=2$ the block invariants are known and most of the conjectures are verified (see \cite{Sambale}), we assume $p>2$ for the rest of the paper. If $D$ is abelian, Brauer's Height Zero Conjecture is true by \cite{KessarMalle} (using the classification).
Hence, we can also assume that $D$ is nonabelian. Then we have to distinguish whether $D$ splits or not. In the nonsplit case the main theorem of \cite{GaoBrauer} says that $B$ is nilpotent. Again, the Height Zero Conjecture holds. Thus, let us assume that $D$ is a nonabelian split metacyclic group. Then $D$ has a presentation of the form
\begin{equation}\label{pres}
D=\langle x,y\mid x^{p^m}=y^{p^n}=1,\ yxy^{-1}=x^{1+p^l}\rangle
\end{equation}
with $0<l<m$ and $m-l\le n$. Many of the results in this paper will depend on these parameters.
Assume that the map $x\to x^{\alpha_1}$ generates an automorphism of $\langle x\rangle$ of order $p-1$. 
Then by Theorem~2.5 in \cite{GaoBrauer} the map $\alpha$ with $\alpha(x)=x^{\alpha_1}$ and $\alpha(y)=y$ is an automorphism of $D$ of order $p-1$. 
By the Schur-Zassenhaus Theorem applied to $\pcore_{p}(\Aut(D))\unlhd\Aut(D)$, $\langle\alpha\rangle$ is unique up to conjugation in $\Aut(D)$. In particular the isomorphism type of the semidirect product $D\rtimes\langle\alpha\rangle$ does not depend on the choice of $\alpha$. 
We denote the inertial quotient of $B$ by $I(B)$; in particular $e(B)=|I(B)|$. It is known that $I(B)$ is a $p'$-subgroup of the outer automorphism group $\Out(D)$.
Hence, we may assume that $I(B)\le\langle\alpha\rangle$. Sometimes we regard $\alpha$ as an element of $\N_G(D)$ by a slight abuse of notation.

We fix a Brauer correspondent $b_D$ of $B$ in $\C_G(D)$.
For an element $u\in D$ we have a $B$-subsection $(u,b_u)\in(D,b_D)$. Here $b_u$ is a Brauer correspondent of $B$ in $\C_G(u)$. Let $\mathcal{F}$ be the fusion system of $B$. Then by Proposition~5.4 in \cite{Stancu}, $\mathcal{F}$ is controlled. In particular $\C_D(u)$ is a defect group of $b_u$ (see Theorem~2.4(ii) in \cite{Linckelmann2}). In case $l(b_u)=1$ we denote the unique irreducible Brauer character of $b_u$ by $\phi_u$. Then the generalized decomposition numbers $d^u_{ij}$ form a vector $d^u:=(d^u_{\chi\phi_u}:\chi\in\Irr(B))$. More generally we have subpairs $(R,b_R)\le(D,b_D)$ for every subgroup $R\le D$. In particular $I(B)=\N_G(D,b_D)/D\C_G(D)$.
For $r\in\mathbb{N}$ we set $\zeta_r:=e^{2\pi i/r}$. 

\begin{Proposition}\label{elem}
Let $B$ be a $p$-block of a finite group with a nonabelian metacyclic defect group for an odd prime $p$. Then $l(B)\ge e(B)$. 
\end{Proposition}
\begin{proof}
We use the notation above. If $D$ is nonsplit, we have $e(B)=l(B)=1$. Thus, assume that $D$ is given by \eqref{pres}.
Let $m(d)$ be the multiplicity of $d\in\mathbb{N}$ as an elementary divisor of the Cartan matrix of $B$. It is well known that $m(p^{m+n})=m(|D|)=1$. Hence, it suffices to show $m(p^n)\ge e(B)-1$.

By Corollary~V.10.12 in \cite{Feit} we have
\[m(p^n)=\sum_{R\in\mathcal{R}}{m_B^{(1)}(R)}\]
where $\mathcal{R}$ is a set of representatives for the $G$-conjugacy classes of subgroups of $G$ of order $p^n$. After combining this with the formula (2S) of \cite{BroueOlsson} we get
\[m(p^n)=\sum_{(R,b_R)\in\mathcal{R}'}{m_B^{(1)}(R,b_R)}\]
where $\mathcal{R}'$ is a set of representatives for the $G$-conjugacy classes of $B$-subpairs $(R,b_R)$ such that $R$ has order $p^n$. 

Thus, it suffices to prove $m_B^{(1)}(\langle y\rangle,b_y)\ge e(B)-1$.
By (2Q) in \cite{BroueOlsson} we have $m^{(1)}_B(\langle y\rangle,b_y)=m^{(1)}_{B_y}(\langle y\rangle)$ where $B_y:=b_y^{\N_G(\langle y\rangle,b_y)}$. 
It is easy to see that $\N_D(\langle y\rangle)=\C_D(y)$, because $D/\langle x\rangle\cong\langle y\rangle$ is abelian.
Since $B$ is controlled and $I(B)$ acts trivially on $\langle y\rangle$, we get $\N_G(\langle y\rangle,b_y)=\C_G(y)$ and $B_y=b_y$. Thus, it remains to prove $m^{(1)}_{b_y}(\langle y\rangle)\ge e(B)-1$. 
Let $x^iy^j\in\C_D(y)\setminus\langle y\rangle$. Then $x^i\in\Z(D)$. Hence, by Theorem~2.3(2)(iii) in \cite{GaoBrauer} we have $\C_D(y)=\Z(D)\langle y\rangle=\langle x^{p^{m-l}}\rangle\times\langle y\rangle$.
By Proposition~2.1(b) in \cite{AnControlled}, also $b_y$ is a controlled block. Observe that $(\C_D(y),b_{\C_D(y)})$ is a maximal $b_y$-subpair. Since $\alpha\in\N_{\C_G(y)}(\C_D(y),b_{\C_D(y)})$, we see that $e(b_y)=e(B)$.

As usual, $b_y$ dominates a block of $\C_G(y)/\langle y\rangle$ with cyclic defect group $\C_D(y)/\langle y\rangle\cong\langle x^{p^{m-l}}\rangle$. Hence, $p^n$ occurs as elementary divisor of the Cartan matrix of $b_y$ with multiplicity $e(b_y)-1=e(B)-1$ (see \cite{Dade,Fujii}). By Corollary~3.7 in \cite{OlssonLDG} every lower defect group of $b_y$ must contain $\langle y\rangle$. This implies $m^{(1)}_{b_y}(\langle y\rangle)=e(B)-1$. 
\end{proof}

Since Alperin's Weight Conjecture would imply that $l(B)=e(B)$, it is reasonable that $\langle y\rangle$ and $D$ are the only (nontrivial) lower defect groups of $D$ up to conjugation. However, we do not prove this. 
We remark that Proposition~\ref{elem} would be false for abelian metacyclic defect groups (see \cite{Kiyota}).

We introduce a general lemma.

\begin{Lemma}\label{controlled}
Let $B$ be a controlled block of a finite group $G$ with Brauer correspondent $b_D$ in $\C_G(D)$. If $(u,b_u)\in(D,b_D)$ is a subsection such that $\N_G(D,b_D)\cap\C_G(u)\subseteq\C_D(u)\C_G(\C_D(u))$, then $e(b_u)=l(b_u)=1$.
\end{Lemma}
\begin{proof}
By Proposition~2.1 in \cite{AnControlled}, $b_u$ is a controlled block with Sylow $b_u$-subpair $(\C_D(u),b_{\C_D(u)})$. Hence,
\[e(b_u)=\lvert\N_{\C_G(u)}(\C_D(u),b_{\C_D(u)})/\C_D(u)\C_G(\C_D(u))\rvert.\]
Every $\mathcal{F}$-automorphism on $\C_D(u)$ is a restriction from $\Aut_{\mathcal{F}}(D)$. This gives
\[\N_{\C_G(u)}(\C_D(u),b_{\C_D(u)})\subseteq(\N_G(D,b_D)\cap\C_G(u))\C_G(\C_D(u))\subseteq\C_D(u)\C_G(\C_D(u)).\]
Thus, we have $e(b_u)=1$. Since $b_u$ is controlled, it follows that $b_u$ is nilpotent and $l(b_u)=1$.
\end{proof}

\begin{Theorem}\label{lowerbound}
Let $B$ be a $p$-block of a finite group with a nonabelian split metacyclic defect group for an odd prime $p$. Then \[k(B)\ge\biggl(\frac{p^l+p^{l-1}-p^{2l-m-1}-1}{e(B)}+e(B)\biggr)p^n.\]
\end{Theorem}
\begin{proof}
If $e(B)=1$, the block $B$ is nilpotent. Then the claim follows from Theorem~2.3(2)(iii) in \cite{GaoBrauer} and Remark~2.4 in \cite{HK}. So, assume $e(B)>1$. The idea is to use Brauer's formula Theorem~5.9.4 in \cite{Nagao}.
Let $u\in D$. Then $b_u$ has metacyclic defect group $\C_D(u)$. Assume first that $u\in\C_D(I(B))$. Since $I(B)$ acts freely on $\langle x\rangle$, we see that $u\in\langle y\rangle$.
As in the proof of Proposition~\ref{elem} (for $u=y$), we get $e(b_u)=e(B)$. If $\C_D(u)$ is nonabelian, Proposition~\ref{elem} implies $l(b_u)\ge e(B)$. Now suppose that $\C_D(u)$ is abelian. Since $y\in\C_D(u)$, it follows that $\C_D(u)=\C_D(y)=\langle x^{p^{m-l}}\rangle\times\langle y\rangle$. Thus, by Theorem~1 in \cite{Watanabe1} we have $l(b_u)=l(b_y)=e(B)$.

Now assume that $u$ is not $\mathcal{F}$-conjugate to an element of $\C_D(I(B))=\langle y\rangle$.
We are going to show that $e(b_u)=l(b_u)=1$ by using Lemma~\ref{controlled}. For this let $\gamma\in(\N_G(D,b_D)\cap\C_G(u))\setminus\C_D(u)\C_G(\C_D(u))$ by way of contradiction. Since $D\C_G(D)\cap\C_G(u)=\C_G(D)\C_D(u)\subseteq\C_D(u)\C_G(\C_D(u))$, $\gamma$ is not a $p$-element.
Hence, after replacing $\gamma$ by a suitable power if necessary, we may assume that $\gamma$ is a nontrivial $p'$-element modulo $\C_G(D)$. 
By the Schur-Zassenhaus Theorem (in our special situation one could use more elementary theorems) applied to $D/\Z(D)\unlhd \Aut_{\mathcal{F}}(D)$, $\gamma$ is $D$-conjugate to a nontrivial power of $\alpha$ (modulo $\C_G(D)$). But then $u$ is $D$-conjugate to an element of $\langle y\rangle$. Contradiction. Hence, we have $\N_G(D,b_D)\cap\C_G(u)\subseteq\C_D(u)\C_G(\C_D(u))$ and $e(b_u)=l(b_u)=1$ by Lemma~\ref{controlled}.

It remains to determine a set $\mathcal{R}$ of representatives for the $\mathcal{F}$-conjugacy classes of $D$ (see Lemma~2.4 in \cite{SambaleD2nastC2m}). 
Since the powers of $y$ are pairwise nonconjugate in $\mathcal{F}$, we get $p^n$ subsections $(u,b_u)$ such that $l(b_u)\ge e(B)$ (including the trivial subsection). 

By Theorem~2.3(2)(iii) in \cite{GaoBrauer} we have $|D'|=p^{m-l}$ and $\lvert\Z(D)\rvert=p^{n-m+2l}$. Hence, Remark~2.4 in \cite{HK} implies that $D$ has precisely $p^{n-m+2l-1}(p^{m-l+1}+p^{m-l}-1)$ conjugacy classes. 
Let $C$ be one of these classes which do not intersect $\langle y\rangle$. 
Assume $\alpha^i(C)=C$ for some $i\in\mathbb{Z}$ such that $\alpha^i\ne 1$. Then there are elements $u\in C$ and $w\in D$ such that $\alpha^i(u)=wuw^{-1}$. Hence $\gamma:=w^{-1}\alpha^i\in\N_G(D,b_D)\cap\C_G(u)$. Since $\gamma$ is not a $p$-element, we get a contradiction as above. This shows that no nontrivial power of $\alpha$ can fix $C$ as a set. Thus, all these conjugacy classes split in 
\[\frac{p^{m-l+1}+p^{m-l}-p^{m-2l+1}-1}{e(B)}p^{n-m+2l-1}\]
orbits of length $e(B)$ under the action of $I(B)$. For every element $u$ in one of these classes we have $l(b_u)=1$ as above. This gives
\[k(B)=\sum_{u\in\mathcal{R}}{l(b_u)}\ge e(B)p^n+\frac{p^l+p^{l-1}-p^{2l-m-1}-1}{e(B)}p^n.\qedhere\]
\end{proof}

The results for blocks with maximal defect in \cite{Gaofull} show that the bound in Theorem~\ref{lowerbound} is sharp (after evaluating the geometric series in Theorem~1.1 of \cite{Gaofull}).

\begin{Theorem}\label{HZC}
Let $B$ be a $p$-block of a finite group with a nonabelian split metacyclic defect group $D$ for an odd prime $p$. Then
\begin{align*}
k_0(B)&\le\biggl(\frac{p^l-1}{e(B)}+e(B)\biggr)p^n\le p^{n+l}=|D:D'|,\\
\sum_{i=0}^{\infty}{p^{2i}k_i(B)}&\le\biggl(\frac{p^l-1}{e(B)}+e(B)\biggr)p^{n+m-l}\le p^{n+m}=|D|,\\
k_i(B)&=0\hspace{5mm}\text{for}\hspace{5mm}i>\min\biggl\{2(m-l),\frac{m+n-1}{2}\biggr\}.
\end{align*}
In particular $k_0(B)<k(B)$, i.\,e. Brauer's Height Zero Conjecture holds for $B$.
\end{Theorem}
\begin{proof}
We consider the subsection $(y,b_y)$. We have already seen that $l(b_y)=e(B)$ and $\C_D(y)/\langle y\rangle$ is cyclic of order $p^l$. Hence, Proposition~2.5(i) in \cite{HKS} implies the first inequality. For the second we consider $u:=x^{p^{m-l}}\in\Z(D)$. Since $u$ is not $D$-conjugate to a power of $y$, the proof of Theorem~\ref{lowerbound} gives $l(b_u)=1$. Moreover, $\lvert\Aut_{\mathcal{F}}(\langle u\rangle)\rvert=e(B)$. Thus, Theorem~4.10 in \cite{HKS} shows the second claim.
Since $k_0(B)>0$, it follows at once that $k_i(B)=0$ for $i>(n+m-1)/2$. On the other hand Corollary~V.9.10 in \cite{Feit} implies $k_i(B)=0$ for $i>2(m-l)$.

Now we discuss the claim $k_0(B)<k(B)$. By Theorem~\ref{lowerbound} it suffices to show
\[\biggl(\frac{p^l-1}{e(B)}+e(B)\biggr)p^n<\biggl(\frac{p^l+p^{l-1}-p^{2l-m-1}-1}{e(B)}+e(B)\biggr)p^n.\]
This reduces to $l<m$, one of our hypotheses. 
\end{proof}

Again for blocks with maximal defect the bound on $k_0(B)$ in Theorem~\ref{HZC} is sharp (see \cite{Gaofull}). On the other hand the bound on the height of the irreducible characters is probably not sharp in general.

\begin{Corollary}\label{kB}
Let $B$ be a $p$-block of a finite group with a nonabelian split metacyclic defect group for an odd prime $p$. Then
\[k(B)\le\biggl(\frac{p^l-1}{e(B)}+e(B)\biggr)(p^{n+m-l-2}+p^n-p^{n-2}).\]
\end{Corollary}
\begin{proof}
In view of Theorem~\ref{HZC}, the number $k(B)$ is maximal if $k_0(B)$ is maximal and $k_1(B)=k(B)-k_0(B)$. Then
\[k_1(B)\le\biggl(\frac{p^l-1}{e(B)}+e(B)\biggr)(p^{n+m-l-2}-p^{n-2})\]
and the result follows.
\end{proof}

Apart from a special case covered in \cite{Schulz}, it seems that there are no results about $B$ in the literature for $p$-solvable groups. We take the opportunity to give such a result which also holds in a more general situation.

\begin{Theorem}
Let $B$ be a controlled block of a $p$-solvable group over an algebraically closed field $F$ of characteristic $p$. If $I(B)$ is cyclic, then $B$ is Morita equivalent to the group algebra $F[D\rtimes I(B)]$ where $D$ is the defect group of $B$. In particular $k(B)=k(D\rtimes I(B))$ and $l(B)=e(B)$.
\end{Theorem}
\begin{proof}
Let $P\unlhd D$, $H$ and $\overline{H}$ as in Theorem~A in \cite{Kpsolv}. As before let $\mathcal{F}$ be the fusion system of $B$. Then part (iii) and (v) of Theorem~A in \cite{Kpsolv} imply that $P$ is $\mathcal{F}$-radical. Moreover, the Hall-Higman Lemma gives \[\C_D(P)\pcore_{p'}(H)/\pcore_{p'}(H)\subseteq\C_{\overline{H}}(\pcore_p(\overline{H}))\subseteq\pcore_p(\overline{H})=P\pcore_{p'}(H)/\pcore_{p'}(H).\]
Since $P$ is normal in $H$, we have $\C_D(P)\subseteq P$. In particular $P$ is also $\mathcal{F}$-centric.
Now let $g\in\N_G(P,b_P)$. Since $B$ is controlled, there exists a $h\in\N_G(D,b_D)$ such that $h^{-1}g\in\C_G(P)$. Hence, $g\in\N_G(D,b_D)\C_G(P)$ and $D\C_G(P)/P\C_G(P)\unlhd\N_G(P,b_P)/P\C_G(P)$. Since $P$ is $\mathcal{F}$-radical, it follows that $P\C_G(P)=D\C_G(P)$. Now $\C_D(P)=\Z(P)$ implies $P=D$. Hence, $\overline{H}\cong D\rtimes I(B)$. 
Observe at this point that $I(B)$ can be regarded as a subgroup of $\Aut(D)$ by the Schur-Zassenhaus Theorem. Moreover, this subgroup is unique up to conjugation in $\Aut(D)$. Hence, the isomorphism type of $D\rtimes I(B)$ is uniquely determined.
Since $I(B)$ is cyclic, the $2$-cocycle $\gamma$ appearing in \cite{Kpsolv} is trivial. Thus, the result follows from Theorem~A(iv).
\end{proof}

Let us consider the opposite situation where $G$ is quasisimple. Then the main theorem of \cite{AnControlled} tells us that $B$ cannot have nonabelian metacyclic defect groups. Thus, in order to settle the general case it would be sufficient to reduce the situation to quasisimple groups.

For the convenience of the reader we collect the results about metacyclic defect groups.

\begin{Theorem}\label{metac}
Let $B$ be a block of a finite group with metacyclic defect group. Then Brauer's $k(B)$-Conjecture, Brauer's Height Zero Conjecture and Olsson's Conjecture are satisfied for $B$.
\end{Theorem}

In the next sections we make restrictions on the parameters $p$, $m$, $n$ and $l$ in order to prove stronger results.

\section{The group $M_{p^{m+1}}$}
Let $n=1$. Then $m=l+1$ and $D$ is the unique nonabelian group of order $p^{m+1}$ with exponent $p^m$. We denote this group by $M_{p^{m+1}}$ (compare with \cite{Holloway}). It follows from Theorem~\ref{HZC} that $k_i(B)=0$ for $i>2$. We will see that the same holds for $i=2$.

\begin{Theorem}\label{k2}
Let $B$ be a block of a finite group with defect group $M_{p^{m+1}}$ where $p$ is an odd prime and $m\ge 2$. Then 
$k_i(B)=0$ for $i\ge 2$. In particular the following conjectures are satisfied for $B$ (in addition to those listed in Theorem~\ref{metac}):
\begin{itemize}
\item Eaton's Conjecture \cite{Eaton}
\item Eaton-Moretó Conjecture \cite{EatonMoreto}
\item Robinson's Conjecture \cite{RobConjecture}
\item Malle-Navarro Conjecture \cite{MalleNavarro}
\end{itemize}
\end{Theorem}
\begin{proof}
Assume $k_2(B)>0$. We are going to show that the following inequality from Theorem~\ref{HZC} is not satisfied:
\begin{equation}\label{contra}
k_0(B)+p^2k_1(B)+p^4k_2(B)\le\biggl(\frac{p^{m-1}-1}{e(B)}+e(B)\biggr)p^2.
\end{equation}
In order to do so, we may assume $k_2(B)=1$. Moreover, taking Theorem~\ref{lowerbound} into account we assume 
\begin{align*}
k_0(B)&=\biggl(\frac{p^{m-1}-1}{e(B)}+e(B)\biggr)p,&k_1(B)&=\frac{p^{m-1}-p^{m-2}}{e(B)}-1.
\end{align*}
Now \eqref{contra} gives the contradiction
\[p^4\le(e(B)+1)p^2-\frac{p^2-p}{e(B)}-e(B)p\le p^3.\]
Hence, $k_2(B)=0$. 
In particular Eaton's Conjecture is in fact equivalent to Brauer's $k(B)$-Conjecture and Olsson's Conjecture. Also the Eaton-Moretó Conjecture is trivially satisfied. Robinson's Conjecture stated in the introduction of \cite{RobConjecture} reads: If $D$ is nonabelian, then $p^{h(\chi)}<|D:\Z(D)|$ for all $\chi\in\Irr(B)$. This is true in our case.
It remains to verify the Malle-Navarro Conjecture. For this observe
\[\frac{k(B)}{l(B)}\le\biggl(\frac{p^{m-1}-1}{e(B)^2}+1\biggr)(p+1-p^{-1})\le p^m+p^{m-1}-p^{m-2}=k(D)\]
by Corollary~\ref{kB} and Remark~2.4 in \cite{HK}. Now we establish a lower bound for $k_0(B)$. 
From Theorem~\ref{HZC} we get 
\[k_1(B)\le\frac{p^{m-1}-1}{e(B)}+e(B)-1.\]
This gives
\begin{equation}\label{k0long}
k_0(B)=k(B)-k_1(B)\ge\frac{p^m-p^{m-2}-p+1}{e(B)}+e(B)(p-1)+1.
\end{equation}
The other inequality of the Malle-Navarro Conjecture reads $k(B)\le k_0(B)k(D')=k_0(B)p$. After a calculation using \eqref{k0long} and Corollary~\ref{kB} this boils down to
\[p^m+2p^{m-1}+p^2\le p^{m+1}+2p+1\]
which is obviously true.
\end{proof}

The argument in the proof of Theorem~\ref{k2} can also be used to improve the general bound for the heights in Theorem~\ref{HZC} at least in some cases. However, it does not suffice to prove $k_i(B)=0$ for $i>m-l$ (which is conjectured). The next theorem also appears in \cite{GaoZeng}.

\begin{Theorem}\label{diff}
Let $B$ be a block of a finite group with defect group $M_{p^{m+1}}$ where $p$ is an odd prime and $m\ge 2$. Then
\[k(B)-l(B)=\frac{p^m+p^{m-1}-p^{m-2}-p}{e(B)}+e(B)(p-1).\]
\end{Theorem}
\begin{proof}
By the proof of Theorem~\ref{lowerbound}, it suffices to show $l(b_u)=e(B)$ for $1\ne u\in\langle y\rangle$. Since $n=1$, we have $\C_D(u)=\Z(D)\langle y\rangle=\langle x^p\rangle\times\langle y\rangle$. Thus, by Theorem~1 in \cite{Watanabe1} we have $l(b_u)=e(B)$.
\end{proof}

This result leads to the distribution of the irreducible characters into $p$-conjugate and $p$-rational characters.
We need this later for the study of decomposition numbers. We denote the Galois group of $\mathbb{Q}(\zeta_{|G|})|\mathbb{Q}(\zeta_{|G|_{p'}})$ by $\mathcal{G}$. Then restriction gives an isomorphism $\mathcal{G}\cong\Gal(\mathbb{Q}(\zeta_{|G|_p})|\mathbb{Q})$. In particular since $p$ is odd, $\mathcal{G}$ is cyclic of order $|G|_p(p-1)/p$. We often identify both groups. 

\begin{Proposition}\label{pcon}
Let $B$ be a block of a finite group with defect group $M_{p^{m+1}}$ where $p$ is an odd prime and $m\ge 2$. Then the ordinary irreducible characters of $B$ split into orbits of $p$-conjugate characters of the following lengths:
\begin{itemize}
\item two orbits of length $p^{m-2}(p-1)/e(B)$
\item one orbit of length $p^i(p-1)/e(B)$ for every $i=0,\ldots,m-3$
\item $(p-1)/e(B)+e(B)$ orbits of length $p-1$
\item $(p-1)/e(B)$ orbits of length $p^i(p-1)$ for every $i=1,\ldots,m-2$
\item $l(B)\,(\ge e(B))$ $p$-rational characters
\end{itemize}
\end{Proposition}
\begin{proof}
By Brauer's Permutation Lemma (Lemma~IV.6.10 in \cite{Feit}) it suffices to reveal the orbits of $\mathcal{G}$ on the columns of the generalized decomposition matrix. The ordinary decomposition numbers are all integral, so the action on these columns is trivial. This gives $l(B)$ $p$-rational characters. Now we consider a set of representatives for the $B$-subsections as in Theorem~\ref{lowerbound}. 

There are $(p^{m-1}-1)/e(B)$ nontrivial major subsections $(z,b_z)$. All of them satisfy $l(b_z)=1$ and $\Aut_{\mathcal{F}}(\langle z\rangle)=I(B)$. So these columns form $m-1$ orbits of lengths $p^{m-2}(p-1)/e(B)$, $p^{m-3}(p-1)/e(B),\ldots,(p-1)/e(B)$ respectively.
Now for $u\in\langle x\rangle\setminus\Z(D)$ we have $l(b_u)=1$ and $\Aut_{\mathcal{F}}(\langle u\rangle)=\langle y\rangle\times I(B)$. This gives another orbit of length $p^{m-2}(p-1)/e(B)$.
Next let $1\ne u\in\langle y\rangle$. Then $l(b_u)=e(B)$ and $\Aut_{\mathcal{F}}(\langle u\rangle)=1$. Hence, we get $e(B)$ orbits of length $p-1$ each. 

Finally let $u:=x^iy^j\in D\setminus\langle x\rangle$ such that $u$ is not conjugate to an element of $\langle y\rangle$. 
As in the proof of Theorem~\ref{lowerbound}, $p^l\nmid i$ holds.
Since $|D'|=p$, we have $(x^iy^j)^p=x^{ip}$ by Hilfssatz~III.1.3 in \cite{Huppert}. In particular $D'\subseteq\langle u\rangle$ and $\N_D(\langle u\rangle)=D$. Moreover, $|D:\Z(D)|=p^2$ and $\lvert\Aut_D(\langle u\rangle)\rvert=p$. 
Since $I(B)$ acts trivially on $D/\langle x\rangle\cong\langle y\rangle$, we see that $\lvert\Aut_{\mathcal{F}}(\langle u\rangle)\rvert=p$. 
The calculation above shows that $u$ has order $p^{m-\log i}$. We have exactly $p^{m-\log i-1}(p-1)^2$ such elements of order $p^{m-\log i}$. These split in $p^{m-\log i-2}(p-1)^2/e(B)$ conjugacy classes. In particular we get $(p-1)/e(B)$ orbits of length $p^{m-i-2}(p-1)$ each for every $i=0,\ldots,l-1=m-2$.
\end{proof}

It should be emphasized that the proof of Proposition~\ref{pcon} heavily relies on the fact $\Aut_{\mathcal{F}}(\langle u\rangle)=1$ whenever $l(b_u)>1$. Since otherwise it would be not clear, whether some Brauer characters of $b_u$ are conjugate under $\N_G(\langle u\rangle,b_u)$. In other words: Generally the knowledge of $k(B)-l(B)$ does not provide the distribution into $p$-conjugate and $p$-rational characters.

For $p=3$ the inequalities Theorem~\ref{lowerbound} and Corollary~\ref{kB} almost coincide. This allows us to prove the Alperin-McKay Conjecture.

\begin{Theorem}\label{AMC}
Let $B$ be a nonnilpotent block of a finite group with defect group $M_{3^{m+1}}$ where $m\ge 2$. Then 
\begin{align*}
e(B)&=2,&k_0(B)&=\frac{3^m+9}{2},\\
k_1(B)&\in\bigl\{3^{m-2},3^{m-2}+1\bigr\},&k_i(B)&=0\text{ for }i\ge 2,\\
k(B)&\in\biggl\{\frac{11\cdot 3^{m-2}+9}{2},\frac{11\cdot 3^{m-2}+11}{2}\biggr\},&l(B)&\in\{2,3\}.
\end{align*}
In particular the Alperin-McKay Conjecture holds for $B$.
\end{Theorem}
\begin{proof}
Since $B$ is nonnilpotent, we must have $e(B)=2$.
From Theorem~\ref{lowerbound} we get $k(B)\ge(11\cdot 3^{m-2}+9)/2$. On the other hand Corollary~\ref{kB} implies $k(B)\le(11\cdot 3^{m-2}+11)/2$. Hence, $l(B)\in\{2,3\}$ by Theorem~\ref{diff}.
Moreover, we have $(3^m+7)/2\le k_0(B)\le (3^m+9)/2$ by Theorem~\ref{HZC} (otherwise $k_1(B)$ would be too large). Now Corollary~1.6 in \cite{Landrock2} shows that $k_0(B)=(3^m+9)/2$. Since we get the same number for the Brauer correspondent of $B$ in $\N_G(D)$, the Alperin-McKay Conjecture follows. 
\end{proof}

The next aim is to show that even Alperin's Weight Conjecture holds in the situation of Theorem~\ref{AMC} provided $m\le 3$. 
Moreover, we verify the Ordinary Weight Conjecture \cite{OWC} in this case using the next proposition.

\begin{Proposition}\label{owc}
Let $B$ be a block of a finite group with defect group $M_{p^{m+1}}$ where $p$ is an odd prime and $m\ge 2$. Then the Ordinary Weight Conjecture for $B$ is equivalent to the equalities 
\begin{align*}
k_0(B)&=\biggl(\frac{p^{m-1}-1}{e(B)}+e(B)\biggr)p,&k_1(B)&=\frac{p-1}{e(B)}p^{m-2}.
\end{align*}
\end{Proposition}
\begin{proof}
We use the version in Conjecture~6.5 in \cite{Kessar}. Let $Q$ be an $\mathcal{F}$-centric and $\mathcal{F}$-radical subgroup of $D$. Since $|D:\Z(D)|=p^2$ and $\C_D(Q)\le Q$, we have $|D:Q|\le p$. Assume $|D:Q|=p$. Then $D/Q\le\Aut_{\mathcal{F}}(Q)$. 
Since $\mathcal{F}$ is controlled, all $\mathcal{F}$-automorphisms on $Q$ come from automorphisms on $D$. In particular $D/Q\unlhd\Aut_{\mathcal{F}}(Q)$. But then $Q$ cannot be $\mathcal{F}$-radical. Hence, we have seen that $D$ is the only $\mathcal{F}$-centric and $\mathcal{F}$-radical subgroup of $D$. It follows that the set $\mathcal{N}_D$ in \cite{Kessar} only consists of the trivial chain. Since $I(B)$ is cyclic, all $2$-cocycles appearing in \cite{Kessar} are trivial. Hence, we see that
\[\textbf{w}(D,d)=\sum_{\chi\in\Irr^d(D)/I(B)}{|I(B)\cap I(\chi)|}\]
where $\Irr^d(D)$ is the set of irreducible characters of $D$ of defect $d\ge 0$ and $I(B)\cap I(\chi):=\{\gamma\in I(B):{^{\gamma}\chi}=\chi\}$. Now the Ordinary Weight Conjecture predicts that $k^d(B)=\textbf{w}(D,d)$ where $k^d(B)$ is the number of irreducible characters of $B$ of defect $d\ge 0$. For $d<m$ both numbers vanish. Now consider $d\in\{m,m+1\}$.
Let us look at a part of the character table of $D$:

\begin{center}
\begin{tabular}{c|ccc}
$D$&$x$&$x^p$&$y$\\\hline
$\chi_{ij}$&$\zeta_{p^{m-1}}^i$&$\zeta_{p^{m-2}}^i$&$\zeta_p^j$\\
$\psi_k$&$0$&$p\zeta_{p^{m-1}}^k$&0
\end{tabular}
\end{center}

Here $i,k\in\{0,\ldots,p^{m-1}-1\}$, $j\in\{0,\ldots,p-1\}$ and $\gcd(k,p)=1$.
The characters of degree $p$ are induced from $\Irr(\langle x\rangle)$. It can be seen that the linear characters of $D$ split into $(p^m-p)/e(B)$ orbits of length $e(B)$ and $p$ stable characters under the action of $I(B)$. This gives 
\[\textbf{w}(D,m+1)=\biggl(\frac{p^{m-1}-1}{e(B)}+e(B)\biggr)p.\] 
Similarly, the irreducible characters of $D$ of degree $p$ split into $p^{m-2}(p-1)/e(B)$ orbits of length $e(B)$. Hence, \[\textbf{w}(D,m)=\frac{p-1}{e(B)}p^{m-2}.\]
The claim follows.
\end{proof}

We introduce another lemma which will be needed at several points.

\begin{Lemma}\label{roots}
Let $q$ be the integral quadratic form corresponding to the Dynkin diagram $A_r$, and let $a\in\mathbb{Z}^r$. 
\begin{enumerate}[(i)]
\item If $q(a)=1$, then $a=\pm(0,\ldots,0,1,1,\ldots,1,0,\ldots,0)$.
\item If $q(a)=2$, then one of the following holds:
\begin{itemize}
\item $a=\pm(0,\ldots,0,1,1,\ldots,1,0,0,\ldots,0,1,1,\ldots,1,0,\ldots,0)$,
\item $a=\pm(0,\ldots,0,1,1,\ldots,1,0,0,\ldots,0,-1,-1,\ldots,-1,0,\ldots,0)$,
\item $a=\pm(0,\ldots,0,1,1,\ldots,1,2,2,\ldots,2,1,1,\ldots,1,0,\ldots,0)$.
\end{itemize}
\end{enumerate}
Here $s,\ldots,s$ includes the possibility of no $s\in\mathbb{Z}$ at all.
\end{Lemma}
\begin{proof}
Without loss of generality $r\ge 2$.
Let $a=(a_1,\ldots,a_r)$. Then
\begin{equation}\label{dynkin}
q(a)=\sum_{i=1}^r{a_i^2}-\sum_{i=1}^{r-1}{a_ia_{i+1}}=\frac{1}{2}\biggl(a_1^2+\sum_{i=1}^{r-1}{(a_i-a_{i+1})^2}+a_r^2\biggr).
\end{equation}
Assume first that $q(a)=1$ and $a_i\ne 0$ for some $i\in\{1,\ldots,r\}$. After replacing $a$ with $-a$ if necessary, we have $a_i>0$. By the equation above we see that the difference between two adjacent entries of $a$ is at most $1$. 
Going from $i$ to the left and to the right, we see that $a$ has the stated form.

Now assume $q(a)=2$. If one of $\{a_1,a_r\}$ is $\pm2$, so must be the other, since each two adjacent entries of $a$ must coincide. But this contradicts Equation~\eqref{dynkin}. Hence, $a_1,a_r\in\{\pm1,0\}$. Now let $|a_i|\ge 3$ for some $i\in\{2,\ldots,r-1\}$. Going from $i$ to the left we get at least two nonvanishing summands in Equation~\eqref{dynkin}. The same holds for the entries on the right side of $i$. Thus, we end up with a configuration where $a_1\ne 0$. This is again a contradiction. It follows that $a_i\in\{\pm1,\pm2,0\}$ for $i=2,\ldots,r-1$. In particular we have only finitely many solutions for $a$. If no $\pm2$ is involved in $a$, it is easy to see that $a$ must be one of the given vectors in the statement of the lemma.
Thus, let us consider $a_i=2$ for some $i\in\{2,\ldots,r-1\}$ (after changing signs if necessary). Then $a_{i-1},a_{i+1}\in\{1,2\}$, since otherwise $(a_i-a_{i-1})^2\ge 4$ or $(a_{i+1}-a_i)^2\ge 4$. Now we can repeat this argument with $a_{i-1}$ and $a_{i+1}$ until we get the desired form for $a$.
\end{proof}

\begin{Theorem}\label{m3}
Let $B$ be a nonnilpotent block of a finite group with defect group $M_{3^{m+1}}$ where $m\in\{2,3\}$. Then 
\begin{align*}
k_0(B)&=\frac{3^m+9}{2},&k_1(B)&=3^{m-2},\\
k(B)&=\frac{11\cdot 3^{m-2}+9}{2},&l(B)&=e(B)=2.
\end{align*}
In particular the following conjectures are satisfied for $B$ (in addition to those listed in previous theorems):
\begin{itemize}
\item Alperin's Weight Conjecture
\item Ordinary Weight Conjecture \cite{OWC}
\end{itemize}
\end{Theorem}
\begin{proof}
Since $B$ is nonnilpotent, we must have $e(B)=2$. The case $m=2$ is very easy and will be handled in the next section together with some more information. Hence, we assume $m=3$ (i.\,e. $|D|=81$) for the rest of the proof. By Theorem~\ref{AMC} we already know $k_0(B)=18$. By way of contradiction we assume $k(B)=22$, $k_1(B)=4$ and $l(B)=3$.

We consider the vector $d^z$ for $z:=x^3\in\Z(D)$. As in \cite{HKS} (we will use this paper a lot) we can write $d^z=\sum_{i=0}^{5}{a_i\zeta_9^i}$ for integral vectors $a_i$. 
We will show that $(a_0,a_i)=0$ for $i\ge 1$. By Lemma~4.7 in \cite{HKS} this holds unless $i=3$. But in this case we have $(a_3,a_3)=0$ and $a_3=0$ by Proposition~4.4 in \cite{HKS}. 
If we follow the proof of Theorem~4.10 in \cite{HKS} closely, it turns out that the vectors $a_i$ are spanned by $a_0$, $a_1$ and $a_4$. So we can also write
\[d^z=a_0+a_1\tau+a_4\sigma\]
where $\tau$ and $\sigma$ are certain linear combinations of powers of $\zeta_9$. Of course, one could give more precise information here, but this is not necessary in this proof. By Lemma~4.7 in \cite{HKS} we have $(a_0,a_0)=27$.

Let $q$ be the quadratic form corresponding to the Dynkin diagram of type $A_3$. We set $a(\chi):=(a_0(\chi),a_1(\chi),\linebreak a_4(\chi))$ for $\chi\in\Irr(B)$. Since the subsection $(z,b_z)$ gives equality in Theorem~4.10 in \cite{HKS}, we have
\[k_0(B)+9k_1(B)=\sum_{\chi\in\Irr(B)}{q(a(\chi))}=54.\]
This implies $q(a(\chi))=3^{2h(\chi)}$ for $\chi\in\Irr(B)$.
Assume that there is a character $\chi\in\Irr(B)$ such that $a_0(\chi)a_i(\chi)>0$ for some $i\in\{1,4\}$. Since $(a_0,a_i)=0$, there must be another character $\chi'\in\Irr(B)$ such that $a_0(\chi')a_i(\chi')<0$. However, then $q(a(\chi'))>3^{2h(\chi)}$ by Lemma~\ref{roots}. Thus, we have shown that $a_0(\chi)a_i(\chi)=0$ for all $\chi\in\Irr(B)$ and $i\in\{1,4\}$. Moreover, if $a_0(\chi)\ne 0$, then $a_0(\chi)=\pm3^{h(\chi)}$ again by Lemma~\ref{roots}.

In the next step we determine the number $\beta$ of integral numbers $d^z(\chi)$ for characters $\chi$ of height $1$. 
Since $(a_0,a_0)=27$, we have $\beta<4$. Let $\psi\in\Irr(B)$ of height $1$ such that $d^z(\psi)\notin\mathbb{Z}$. Then we can form the orbit of $d^z(\psi)$ under the Galois group $\mathcal{H}$ of $\mathbb{Q}(\zeta_9)|(\mathbb{Q}(\zeta_9)\cap\mathbb{R})$. The length of this orbit must be $\lvert\mathcal{H}\rvert=3$. In particular $\beta=1$.

This implies that $d^z(\chi)=a_0(\chi)=\pm1$ for all $18$ characters $\chi\in\Irr(B)$ of height $0$. In the following we derive a contradiction using the orthogonality relations of decomposition numbers. In order to do so, we repeat the argument with the subsection $(x,b_x)$. Again we get equality in Theorem~4.10, but this time for $k_0(B)$ instead of $k_0(B)+9k_1(B)$. Hence, $d^x(\chi)=0$ for characters $\chi\in\Irr(B)$ of height $1$. Again we can write $d^x=\sum_{i=0}^{17}{\overline{a_i}\zeta_{27}^i}$ where $\overline{a_i}$ are integral vectors. Lemma~4.7 in \cite{HKS} implies $(\overline{a}_0,\overline{a}_0)=9$. 
Using Lemma~\ref{roots} we also have $\overline{a}_0(\chi)\in\{0,\pm1\}$ in this case. This gives the final contradiction $0=(d^z,d^x)=(a_0,\overline{a}_0)\equiv 1\pmod{2}$. 

Hence, we have proved that $k(B)=21$, $k_1(B)=3$ and $l(B)=2$. Since $B$ is controlled, Alperin's Weight Conjecture asserts that $l(B)=l(b)$ where $b$ is the Brauer correspondent of $B$ in $\N_G(D)$. Since $e(b)=e(B)$, the claim follows at once. 
The Ordinary Weight Conjecture follows from Proposition~\ref{owc}. 
This completes the proof.
\end{proof}

\section{The group $p^{1+2}_-$}
In this section we restrict further to the case $n=1$ and $m=2$, i.\,e.
\[D=\langle x,y\mid x^{p^2}=y^p=1,\ yxy^{-1}=x^{1+p}\rangle\]
is extraspecial of order $p^3$ and exponent $p^2$. We denote this group by $p^{1+2}_-$ (compare with \cite{Hendren2}). In particular we can use the results from the last section. One advantage of this restriction is that the bounds are slightly sharper than in the general case.

\subsection{Inequalities}
Our first theorem says that the block invariants fall into an interval of length $e(B)$.

\begin{Theorem}\label{inequ}
Let $B$ be a block of a finite group with defect group $p^{1+2}_-$ for an odd prime $p$. Then
\begin{align*}
\frac{p^2-1}{e(B)}+e(B)p\le k(B)&\le\frac{p^2-1}{e(B)}+e(B)p+e(B)-1,\\
\biggl(\frac{p-1}{e(B)}+e(B)\biggr)p-e(B)+1\le k_0(B)&\le\biggl(\frac{p-1}{e(B)}+e(B)\biggr)p,\\
\frac{p-1}{e(B)}\le k_1(B)&\le \frac{p-1}{e(B)}+e(B)-1,\\
k_i(B)&=0\text{ for }i\ge 2,\\[2mm]
e(B)\le l(B)&\le 2e(B)-1,\\
k(B)-l(B)&=\frac{p^2-1}{e(B)}+(p-1)e(B).
\end{align*}
\end{Theorem}
\begin{proof}
The formula for $k(B)-l(B)$ comes from Theorem~\ref{diff}. The lower bound for $l(B)$ and $k(B)$ follow from Proposition~\ref{elem} and Theorem~\ref{lowerbound}. The upper bound for $k_0(B)$ comes from Theorem~\ref{HZC}. 
The same theorem gives also
\[k_1(B)\le \frac{p-1}{e(B)}+e(B)-1.\]
Adding this to the upper bound for $k_0(B)$ results in the stated upper bound for $k(B)$.
Now the upper bound for $l(B)$ follows from $k(B)-l(B)$.
A lower bound for $k_0(B)$ is given by
\[k_0(B)=k(B)-k_1(B)\ge\frac{p^2-1}{e(B)}+e(B)p-\frac{p-1}{e(B)}-e(B)+1=\biggl(\frac{p-1}{e(B)}+e(B)\biggr)p-e(B)+1\]
Moreover,
\[k_1(B)=k(B)-k_0(B)\ge\frac{p^2-1}{e(B)}+e(B)p-\biggl(\frac{p-1}{e(B)}+e(B)\biggr)p=\frac{p-1}{e(B)}.\qedhere\]
\end{proof}

Since we already know that the upper bound for $k_0(B)$ and the lower bound for $k(B)$ are sharp (for blocks with maximal defect), it follows at once that the lower bound for $k_1(B)$ in Theorem~\ref{inequ} is also sharp (compare with Proposition~\ref{owc}).

If $e(B)$ is as large as possible, we can prove slightly more.

\begin{Proposition}\label{cor}
Let $B$ be a block of a finite group with defect group $p^{1+2}_-$ for an odd prime $p$.
If $e(B)=p-1$, then $k(B)\le p^2+p-2$, $l(B)\le 2e(B)-2$ and $k_0(B)\ne p^2-r$ for $r\in\{1,2,4,5,7,10,13\}$.
\end{Proposition}
\begin{proof}
By way of contradiction, assume $k(B)=p^2+p-1$. By Theorem~\ref{inequ} we have $k_0(B)=p^2$ and $k_1(B)=p-1$. Set $z:=x^p\in\Z(D)$. Then we have $l(b_z)=1$. Since $I(B)$ acts regularly on $\Z(D)\setminus\{1\}$, the vector $d^z$ is integral. By Lemma~4.1 in \cite{HKS} we have
$0\ne d^z_{\chi\phi_z}\equiv 0\pmod{p}$ for characters $\chi$ of height $1$. Hence, $d^z$ must consist of $p^2$ entries $\pm1$ and $p-1$ entries $\pm p$. 
Similarly $l(b_x)=1$. Moreover, all powers $x^i$ for $(i,p)=1$ are conjugate under $\mathcal{F}$. Hence, also the vector $d^x$ is integral. 
Thus, the only nonvanishing entries of $d^x$ are $\pm 1$ for the characters of height $0$, because $(d^x,d^x)=p^2$ (again using Lemma~4.1 in \cite{HKS}). Now the orthogonality relations imply the contradiction $0=(d^z,d^x)\equiv 1\pmod{2}$, since $p$ is odd. Thus, we must have $k(B)\le p^2+p-2$ and $l(B)\le 2e(B)-2$.

We have seen that every character of height $0$ corresponds to a nonvanishing entry in $d^x$. If we have a nonvanishing entry for a character of height $1$ on the other hand, then Theorem~V.9.4 in \cite{Feit} shows that this entry is $\pm p$. However, this contradicts the orthogonality relation $(d^z,d^x)=0$. Hence, the nonvanishing entries of $d^x$ are in one-to-one correspondence to the irreducible characters of height $0$. Thus, we see that $p^2$ is the sum of $k_0(B)$ nontrivial integral squares. This gives the last claim. 
\end{proof}

Since in case $e(B)=2$ the inequalities are very strong, it seems reasonable to obtain more precise information here. In the last section we proved for arbitrary $m$ that the Alperin-McKay Conjecture holds provided $p=3$. As a complementary result we now show the same for all $p$, but with the restrictions $m=2$ and $e(B)=2$. We even obtain a refinement of the Alperin-McKay Conjecture which is called the Galois-Alperin-McKay Conjecture (see Conjecture~D in \cite{IsaacsNavarro}).

\begin{Theorem}\label{AMC2}
Let $B$ be a block of a finite group with defect group $p^{1+2}_-$ for an odd prime $p$ and $e(B)=2$. Then $k_0(B)=p(p+3)/2$.
In particular the Galois-Alperin-McKay Conjecture holds for $B$.
\end{Theorem}
\begin{proof}
By Theorem~\ref{AMC} we may assume $p>3$. For some subtle reasons we also have to distinguish between $p=7$ and $p\ne 7$. 
Let us assume first that $p\ne 7$.
By Theorem~\ref{inequ} we have $k_0(B)\in\{p(p+3)/2-1,p(p+3)/2\}$.
We write $d^x=\sum_{i=0}^{p(p-1)-1}{\overline{a}_i\zeta_{p^2}^i}$ with integral vectors $\overline{a}_i$. As in Proposition~4.9 of \cite{HKS} we see that $\overline{a}_i=0$ if $(i,p)=1$. Moreover, the arguments in the proof of Proposition~4.8 of the same paper tell us that $\overline{a}_p=0$ and $\overline{a}_{ip}=\overline{a}_{(p-i)p}$ for $i=2,\ldots,(p-1)/2$. Now let $\tau_i:=\zeta_p^i+\zeta_p^{-i}$ for $i=2,\ldots,(p-1)/2$. Then we can write 
\[d^x=a_0+\sum_{i=2}^{(p-1)/2}{a_i\tau_i}\]
for integral vectors $a_i$. Here observe that $d^x$ is real, since $(x,b_x)$ and $(x^{-1},b_{x^{-1}})$ are conjugate under $I(B)$. By Lemma~4.7 in \cite{HKS} we have $(a_0,a_0)=3p$, $(a_i,a_j)=p$ for $i\ne j$ and $(a_i,a_i)=2p$ for $i\ge 2$. 
Now let $a(\chi)=(a_i(\chi):i=0,2,3,\ldots,(p-1)/2)$ for $\chi\in\Irr(B)$. Moreover, let $q$ be the integral quadratic form corresponding to the Dynkin diagram of type $A_{(p-1)/2}$. Then as in Proposition~4.2 in \cite{HKS} we get
\[\sum_{\chi\in\Irr(B)}{q(a(\chi))}=p\biggl(3+2\frac{p-3}{2}-\frac{p-3}{2}\biggr)=p\frac{p+3}{2}.\]
Let $\chi\in\Irr(B)$ be a character of height $1$. Suppose that $a(\chi)\ne 0$. Then we have $k_0(B)=p(p+3)/2-1$ and $\chi$ is the only character of height $1$ such that $a(\chi)\ne 0$. In particular $\chi$ is $p$-rational and $a(\chi)=a_0(\chi)\in\mathbb{Z}$. Now Theorem~V.9.4 in \cite{Feit} implies $p\mid a_0(\chi)$. Since $(a_0,a_0)=3p$, this gives $p=3$ which contradicts our hypothesis. Hence, we have shown that $a(\chi)=0$ for all characters $\chi\in\Irr(B)$ of height $1$. In particular
\[\sum_{\substack{\chi\in\Irr(B),\\h(\chi)=0}}{q(a(\chi))}=p\frac{p+3}{2}.\]

By way of contradiction suppose that $k_0(B)=p(p+3)/2-1$. Then there is exactly one character $\chi\in\Irr(B)$ such that $q(a(\chi))=2$ (this already settles the case $p=5$). Now the idea is to show that there is a $p$-conjugate character $\psi$ also satisfying $q(a(\psi))>1$. In order to do so, we discuss the different cases in Lemma~\ref{roots}. Here we can of course choose the sign of $a(\chi)$. 

First assume $a(\chi)=(0,\ldots,0,1,1,\ldots,1,0,0,\ldots,0,-1,-1,\ldots,-1,0,\ldots,0)$. Choose an index $k$ corresponding to one of the $-1$ entries in $a(\chi)$. Let $k'\in\{2,\ldots,(p-1)/2\}$ such that $kk'\equiv \pm1\pmod{p}$, and let $\gamma_{k'}\in\mathcal{G}$ be the Galois automorphism which sends $\zeta_p$ to $\zeta_p^{k'}$. Then 
\[\gamma_{k'}(\tau_k)=-1-\sum_{i=2}^{(p-1)/2}{\tau_i}.\]
Apart from this, $\gamma_{k'}$ acts as a permutation on the remaining indices $\{2,\ldots,(p-1)/2\}\setminus\{k\}$. 
This shows that $a(\gamma_{k'}(\chi))$ contains an entry $2$. 
In particular $\gamma_{k'}(\chi)\ne\chi$. Moreover, Lemma~\ref{roots} gives $q(a(\gamma_{k'}(\chi)))>1$.

Next suppose that $a(\chi)=(0,\ldots,0,1,1,\ldots,1,2,2,\ldots,2,1,1,\ldots,1,0,\ldots,0)$. Here we choose $k$ corresponding to an entry $2$ in $a(\chi)$. Then the same argument as above implies that $a(\gamma_{k'}(\chi))$ has a $-2$ on position $k_1$. Contradiction. 

Now let $a(\chi)=(0,0,\ldots,0,1,1,\ldots,1,0,0,\ldots,0,1,1,\ldots,1,0,\ldots,0)$ (observe the leading $0$).
We choose the index $k$ corresponding to a $1$ in $a(\chi)$. Let $\gamma_{k'}$ be the automorphism as above. Observe that $\chi$ is not $p$-rational. Thus, Proposition~\ref{pcon} implies $\gamma_{k'}(\chi)\ne\chi$.
In particular $q(a(\gamma_{k'}(\chi)))=1$. Hence, we must have $a(\gamma_{k'}(\chi))=(-1,-1,\ldots,-1,0,0,\ldots,0)$
where the number of $-1$ entries is uniquely determined by $a(\chi)$. In particular $a(\gamma_{k'}(\chi))$ is independent of the choice of $k$. Now choose another index $k_1$ corresponding to an entry $1$ in $a(\chi)$ (always exists). Then we see that $a(\chi)$ and thus $\chi$ is fixed by $\gamma_{k'}^{-1}\gamma_{k_1'}$. Proposition~\ref{pcon} shows that $\gamma_{k'}^{-1}\gamma_{k_1'}$ must be (an extension of) the complex conjugation. This means $k'\equiv -k_1'\pmod{p}$ and $k\equiv -k_1\pmod{p}$. However this contradicts $2\le k,k_1\le(p-1)/2$.

Finally let $a(\chi)=(1,1,\ldots,1,0,0,\ldots,0,1,1,\ldots,1,0,\ldots,0)$. Here a quite similar argument shows that $a(\chi)$ only contains one entry $0$, say on position $k$. Now we can use the same trick where $k_1\ge 2$ corresponds to an entry $1$. Here $a(\gamma_{k_1'}(\chi))=(0,0,\ldots,0,-1,-1,0,\ldots,0)$. 
Let $k_2\in\{2,\ldots,(p-1)/2\}$ such that $k_2\equiv \pm kk_1'\pmod{p}$. Then the $-1$ entries of $a(\gamma_{k_1'}(\chi))$ lie on positions $k_1'$ and $k_2$. Since these entries lie next to each other, we get $k\pm 1\equiv \pm k_1\pmod{p}$ where the signs are independent. However, this shows that $k$ and $k_1$ are adjacent. Hence, we proved that $a(\chi)=(1,0,1)$ and $p=7$ ($(1,1,0,1)$ is not possible, since $9$ is not a prime). However, this case was excluded. 
Thus, $k_0(B)=p(p+3)/2$.

It remains to deal with the case $p=7$. 
It can be seen that there is in fact a permissible configuration for $k_0(B)=34$:
\[d^x=(\underbrace{1,\ldots,1}_{13\text{ times}},\underbrace{1+\tau_2+\tau_3,\ldots,1+\tau_2+\tau_3}_{6\text{ times}},1+\tau_2,1+\tau_3,\tau_2+\tau_3,\underbrace{\tau_2,\ldots,\tau_2}_{6\text{ times}},\underbrace{\tau_3,\ldots,\tau_3}_{6\text{ times}},0,\ldots,0).\]
Hence, we consider $d^z$ for $z:=x^7$. Suppose by way of contradiction that $k_0(B)=34$. Then $k_1(B)=4$ and $k(B)=38$. By Proposition~\ref{pcon} we have exactly two $7$-rational irreducible characters in $\Irr(B)$. Moreover, the orbit lengths of the $7$-conjugate characters are all divisible by $3$. Hence, we have precisely one $7$-rational character of height $1$ and one of height $0$. In the same way as above we can write $d^z=a_0+a_2\tau_2+a_3\tau_3$ (see Proposition~4.8 in \cite{HKS}). Then $(a_0,a_0)=3\cdot 7^2$, $(a_i,a_j)=7^2$ for $i\ne j$ and $(a_i,a_i)=2\cdot 7^2$ for $i=2,3$. For a character $\chi\in\Irr(B)$ of height $1$ we have $7\mid a_i(\chi)$ for $i=0,2,3$ by Lemma~4.1 in \cite{HKS}. Since
\[\sum_{\chi\in\Irr(B)}{q(a(\chi))}=5p^2,\]
it follows that $q(a(\chi))=7^2$ for every character $\chi\in\Irr(B)$ of height $1$. It is easy to see that $a(\chi)\notin\{\pm7(0,1,1),\linebreak\pm7(1,1,0)\}$. Hence, the four rows $a(\chi)$ for characters $\chi$ of height $1$ have to following form up to signs and permutations: 
\[7\begin{pmatrix}
1&.&.\\
.&1&.\\
.&.&1\\
1&1&1
\end{pmatrix}.\]
Thus, for a character $\chi_i\in\Irr(B)$ of height $0$ ($i=1,\ldots,34$) we have $d^z(\chi_i)=a_0(\chi_i)\ne 0$ and 
\[\sum_{i=1}^{34}{a_0(\chi_i)^2}=7^2.\]
Up to signs and permutations we get $(a_0(\chi_i))=(4,1,\ldots,1)$ (taking into account that only $\chi_1$ can be $7$-rational).
So still no contradiction.

Now consider $d^y_{ij}$. The Cartan matrix of $b_y$ is $7\bigl(\begin{smallmatrix}
4&3\\3&4
\end{smallmatrix}\bigr)$ up to basic sets (see \cite{Dade,Rouquiercyclic}). We can write $d^y_{\chi\phi_1}=\sum_{i=0}^{5}{\widetilde{a}_i(\chi)\zeta_7^i}$ and $d^y_{\chi\phi_2}=\sum_{i=0}^{5}{\widetilde{b}_i(\chi)\zeta_7^i}$ for $\chi\in\Irr(B)$. It follows that $(\widetilde{a}_0,\widetilde{a}_0)=(\widetilde{b}_0,\widetilde{b}_0)=8$ 
(this is basically the same calculation as in Proposition~4.8 in \cite{HKS}). By Corollary~1.15 in \cite{Murai}  
we have $\widetilde{a}_0(\chi_1)\ne 0$ or $\widetilde{b}_0(\chi_1)\ne 0$. With out loss of generality assume $\widetilde{a}_0(\chi_1)\ne 0$. Then $\widetilde{a}_0(\chi_1)=\pm1$, since $(a_0,\widetilde{a}_0)=(d^z,\widetilde{a}_0)=0$.
On the other hand $\widetilde{a}_0(\chi)=0$ for characters $\chi\in\Irr(B)$ of height $1$, because we have equality in Theorem~2.4 in \cite{HKS}.
However, this gives the following contradiction
\[0=(a_0,\widetilde{a}_0)=\sum_{i=1}^{34}{a_0(\chi_i)\widetilde{a}_0(\chi_i)}\equiv\sum_{i=2}^{34}{\widetilde{a}_0(\chi_i)}\equiv\sum_{i=2}^{34}{\widetilde{a}_0(\chi_i)^2}\equiv 7\pmod{2}.\]

Altogether we have proved that $k_0(B)=p(p+3)/2$ for all odd primes $p$. In order to verify the Galois-Alperin-McKay Conjecture we have to consider a $p$-automorphism $\gamma\in\mathcal{G}$. By Lemma~IV.6.10 in \cite{Feit} it suffices to compute the orbits of $\langle\gamma\rangle$ on the columns of the generalized decomposition matrix. For an element $u\in D$ of order $p$, $\gamma$ acts trivially on $\langle u\rangle$. If $u$ has order $p^2$, then $\gamma$ acts as $D$-conjugation on $\langle u\rangle$. This shows that $\gamma$ acts in fact trivially on the columns of the generalized decomposition matrix. In particular all characters of height $0$ are fixed by $\gamma$. Hence, the Galois-Alperin-McKay Conjecture holds.
\end{proof}

\subsection{The case $p\le 11$}
We already know $k_0(B)$ if $e(B)=2$. For small primes it is also possible to obtain $k(B)$.

\begin{Theorem}\label{p5711}
Let $B$ be a block of a finite group with defect group $p^{1+2}_-$ for $3\le p\le 11$ and $e(B)=2$. Then 
\begin{align*}
k(B)=\frac{p^2+4p-1}{2},&&k_0(B)=\frac{p+3}{2}p,&&k_1(B)=\frac{p-1}{2},&&l(B)=2.
\end{align*}
The irreducible characters split into two orbits of $(p-1)/2$ $p$-conjugate characters, $(p+3)/2$ orbits of length $p-1$, and two $p$-rational characters. For $p\ge 5$ the $p$-rational characters have height $0$. 
In particular the following conjectures are satisfied for $B$ (in addition to those listed in previous theorems):
\begin{itemize}
\item Alperin's Weight Conjecture
\item Robinson's Ordinary Weight Conjecture \cite{OWC}
\end{itemize}
\end{Theorem}
\begin{proof}
We have $k_0(B)=p(p+3)/2$ by Theorem~\ref{AMC2}. 
For $p=3$ the block invariants and the distribution into $3$-conjugate and $3$-rational characters follow at once from Theorem~\ref{inequ} and Proposition~\ref{cor}. So we may assume $p>3$ for the first part of the proof.
Suppose $k(B)=(p^2+4p+1)/2$ and $k_1(B)=(p+1)/2$. Then $\Irr(B)$ contains exactly three $p$-rational characters. Moreover, the orbit lengths of the $p$-conjugate characters are all divisible by $(p-1)/2$. Let $z:=x^p$. Then we can write \[d^z=a_0+\sum_{i=2}^{(p-1)/2}{a_i\tau_i}\] 
as in Theorem~\ref{AMC2} where $\tau_i:=\zeta_p^i+\zeta_p^{-i}$ for $i=2,\ldots,(p-1)/2$ (see Proposition~4.8 in \cite{HKS}). Then $(a_0,a_0)=3p^2$, $(a_i,a_j)=p^2$ for $i\ne j$ and $(a_i,a_i)=2p^2$ for $i\ge 2$. For a character $\chi\in\Irr(B)$ of height $1$ we have $p\mid a_i(\chi)$ by Lemma~4.1 in \cite{HKS}. Since
\[\sum_{\chi\in\Irr(B)}{q(a(\chi))}=\frac{p+3}{2}p^2,\]
we have $q(a(\chi))=p^2$ for every character $\chi\in\Irr(B)$ of height $1$. If all characters of height $1$ are $p$-rational, we have $p=5$. But then $(a_0,a_2)=0$. Hence, exactly one character of height $1$ is $p$-rational.
Now choose a non-$p$-rational character $\psi\in\Irr(B)$ of height $1$. Assume $a(\psi)=p(0,\ldots,0,1,1,1,\ldots,1,0,\ldots,0)$ with at least two entries $1$ in a row and at least one entry $0$ (see Lemma~\ref{roots}). 

If $a_0(\psi)=0$, then $a(\gamma(\psi))=p(-1,-1,\ldots,-1,0,0,\ldots,0)=a(\gamma'(\psi))$ for two different Galois automorphisms $\gamma,\gamma'\in\mathcal{G}$ (see proof of Theorem~\ref{AMC2}). Moreover, $\gamma^{-1}\gamma'$ is not (an extension of) the complex conjugation. In particular $(\gamma^{-1}\gamma')(\psi)\ne\psi$. 
Since $(a_2,a_2)=2p^2$, $\gamma^{-1}\gamma'$ (up to complex conjugation) is the only nontrivial automorphism fixing $d^z(\psi)$. So, $(\gamma^{-1}\gamma')^2$ is (an extension of) the complex conjugation. This gives $4\mid p-1$ and $p=5$ again. But for $5$ the whole constellation is not possible, since $a(\psi)$ is $2$-dimensional in this case.

Finally assume $a(\psi)=p(1,1,1,\ldots,1,0,0,\ldots,0)$. Then we can find again an Galois automorphism $\gamma$ (corresponding to an entry $0$ in $a(\psi)$) such that $a(\gamma(\psi))=a(\psi)$. So we get the same contradiction in this case too.

Hence, we have seen that $a(\psi)$ contains either one or $(p-1)/2$ entries $\pm1$. 
Thus, the rows $a(\chi)$ for characters $\chi$ of height $1$ have to following form up to signs and permutations: 
\[p\begin{pmatrix}
1&.&\cdots&.\\
.&1&\ddots&\vdots\\
\vdots&\ddots&\ddots&.\\
.&\cdots&.&1\\
1&1&\cdots&1
\end{pmatrix}.\]
In particular $d^z(\chi_i)=a_0(\chi_i)\ne 0$ for all characters $\chi_i$ of height $0$ ($i=1,\ldots,p(p+3)/2$). Moreover,
\[\sum_{i=1}^{p(p+3)/2}{a_0(\chi_i)^2}=p^2.\]
Subtracting $p(p+3)/2$ on both sides gives
\begin{equation}\label{k0}
\sum_{i=2}^{\infty}{r_i(i^2-1)}=p\frac{p-3}{2}
\end{equation}
for some $r_i\ge 0$. Choose $r_i'\in\{0,1,\ldots,(p-3)/2\}$ such that $r_i\equiv r_i'\pmod{(p-1)/2}$. Since we have only two $p$-rational characters of height $0$, the following inequality is satisfied:
\[\sum_{i=2}^{\infty}{r_i'}\le 2.\]
Using this, it turns out that Equation~\eqref{k0} has no solution unless $p>11$. 
Hence, $k(B)=(p^2+4p-1)/2$. 

The orbit lengths of $p$-conjugate characters follow from Proposition~\ref{pcon}. If there is a $p$-rational character of height $1$, we must have $p=5$. Then of course both characters $\psi_1$, $\psi_2$ of height $1$ must be $5$-rational. For these characters we have $d^z(\psi_i)=a_0(\psi_i)=\pm5$ with the notation above. Now our aim is to show that $\psi_1-\psi_2$ or $\psi_1+\psi_2$ vanishes on the $5$-singular elements of $G$. This is true for the elements in $\Z(D)$. Now let $(u,b_u)$ be a nonmajor $B$-subsection. 
Assume first that $u\in\langle y\rangle$. Since $l(b_u)=2$, we have equality in Theorem~2.4 in \cite{HKS}. This implies $d^u_{\psi_i,\phi_j}=0$ for $i,j\in\{1,2\}$. Next suppose $u\in\langle x\rangle$. Then $d^u(\psi_i)\in\mathbb{Z}$. Hence, Theorem~V.9.4 in \cite{Feit} implies $5\mid d^u(\psi_i)$. Since the scalar product of the integral part of $d^u$ is $15$ (compare with proof of Theorem~\ref{AMC2}), we get $d^u(\psi_i)=0$ for $i=1,2$ again. It remains to handle the case $u\notin\langle x\rangle$ and $l(b_u)=1$. Here Lemma~4.7 in \cite{HKS} shows that the scalar product of the integral part of $d^u$ is $10$. So by the same argument as before $d^u(\psi_i)=0$ for $i=1,2$.
Hence, we have shown that $\psi_1-\psi_2$ or $\psi_1+\psi_2$ vanishes on the $5$-singular elements of $G$. Now Robinson wrote on the second page of \cite{RobObstruction} that under these circumstances the number $2$ is representable by the quadratic form of the Cartan matrix $C$ of $B$. However, by (the proof of) Proposition~\ref{elem}, the elementary divisors of $C$ are $5$ and $5^3$. In particular every entry of $C$ is divisible by $5$. So this cannot happen. Hence, we have shown that the two irreducible characters of height $1$ are $5$-conjugate.

Now let $3\le p\le 11$ be arbitrary. Then the two conjectures follow as usual.
\end{proof}

If we have $p=13$ in the situation of Theorem~\ref{p5711}, then Equation~\eqref{k0} has the solution $r_2=19$, $r_3=1$ and $r_i=0$ for $i\ge 4$. 
For larger primes we get even more solutions.
With the help of Theorem~\ref{m3} and Theorem~\ref{p5711} it is possible to obtain $k(B)-l(B)$ in the following situations:
\begin{itemize}
\item $p=3$, $D$ as in \eqref{pres} with $n=l=2$ (in particular $|D|\le 3^6$),
\item $3\le p\le 11$, $D$ as in \eqref{pres} with $n=2$ and $l=1$ (in particular $|D|\le p^5$), and $e(B)=2$.
\end{itemize} 
However, there is no need to do so.

In case $p=3$, Theorem~\ref{p5711} applies to all nonnilpotent blocks. Here we can show even more.

\begin{Theorem}\label{p3}
Let $B$ be a nonnilpotent block of a finite group with defect group $3^{1+2}_-$. Then $e(B)=l(B)=2$, $k(B)=10$, $k_0(B)=9$ and $k_1(B)=1$. There are three pairs of $3$-conjugate irreducible characters (of height $0$) and four $3$-rational irreducible characters. The Cartan matrix of $B$ is given by
\[3\begin{pmatrix}2&1\\1&5\end{pmatrix}\]
up to basic sets.
Moreover, the Gluing Problem \cite{gluingprob} for $B$ has a unique solution.
\end{Theorem}
\begin{proof}
Since $B$ is nonnilpotent, we get $e(B)=2$.
It remains to show the last two claims.

It is possible to determine the Cartan matrix $C$ of $B$ by enumerating all decomposition numbers with the help of a computer. However, we give a more theoretical argument which does not rely on computer calculations. By (the proof of) Proposition~\ref{elem} $C$ has elementary divisors $3$ and $27$.
Hence, $\widetilde{C}:=\frac{1}{3}C=\bigl(\begin{smallmatrix}a& b\\b&c\end{smallmatrix}\bigr)$ is an integral matrix with elementary divisors $1$ and $9$. We may assume that $\widetilde{C}$ is reduced as binary quadratic form by changing the basic set if necessary (see \cite{Buell}). This means $0\le 2b\le a\le c$. We derive $3a^2/4\le ac-b^2=\det \widetilde{C}=9$ and $a\in\{1,2,3\}$. This gives only the following two possibilities for $\widetilde{C}$:
\begin{align*}
\begin{pmatrix}2&1\\1&5\end{pmatrix},&&\begin{pmatrix}1&0\\0&9\end{pmatrix}.
\end{align*}
It remains to exclude the second matrix. So assume by way of contradiction that this matrix occurs for $\widetilde{C}$. Let $d_1$ be the column of decomposition numbers corresponding to the first irreducible Brauer character in $B$. Then $d_1$ consists of three entries $1$ and seven entries $0$. 

It can be seen easily that
$d^x=(1,\ldots,1,0)^{\text{T}}$ up to permutations and signs. Since $(d_1,d^x)=0$, we have $d_1(\chi_{10})=1$ where $\chi_{10}$ is the unique irreducible character of height $1$.

Now consider $y$. 
The Cartan matrix of $b_y$ is $3\left(\begin{smallmatrix}2&1\\1&2\end{smallmatrix}\right)$ (see \cite{Dade,Rouquiercyclic}). 
We denote the two irreducible Brauer characters of $b_y$ by $\phi_1$ and $\phi_2$ and write $d^y_{\chi\phi_i}=a_i(\chi)+b_i(\chi)\zeta_3$ for $i=1,2$. Then we have
\begin{align*}
6&=(a_i,a_i)+(b_i,b_i)-(a_i,b_i),\\
0&=(a_i,a_i)+2(a_i,b_i)\zeta_3+(b_i,b_i)\overline{\zeta_3}=(a_i,a_i)-(b_i,b_i)+(2(a_i,b_i)-(b_i,b_i))\zeta_3,\\
3&=(a_1,a_2)+(b_1,b_2)+(b_1,a_2)\zeta_3+(a_1,b_2)\overline{\zeta_3}=(a_1,a_2)+(b_1,b_2)-(a_1,b_2)+((b_1,a_2)-(a_1,b_2))\zeta_3,\\
0&=(a_1,a_2)+((a_1,b_2)+(b_1,a_2))\zeta_3+(b_1,b_2)\overline{\zeta_3}=(a_1,a_2)-(b_1,b_2)+((a_1,b_2)+(b_1,a_2)-(b_1,b_2))\zeta_3.
\end{align*}
Thus, $(a_i,a_i)=(b_i,b_i)=4$, $(a_i,b_i)=(a_1,a_2)=(b_1,b_2)=2$ and $(a_1,b_2)=(a_2,b_1)=1$ for $i=1,2$. It follows that the numbers $d^y_{\chi\phi_i}$ can be given in the following form (up to signs and permutations):
\[
\left(\begin{array}{cccccccccc}
1&1&1+\zeta_3&1+\zeta_3&\zeta_3&\zeta_3&.&.&.&.\\
1&.&1+\zeta_3&.&\zeta_3&.&1+\zeta_3&1&\zeta_3&.
\end{array}\right)^{\text{T}}.\]
But now we see that $d_1$ cannot be orthogonal to both of these columns. This contradiction gives $C$ up to basic sets.

Finally we investigate the Gluing Problem for $B$. For this we use the notation of \cite{Parkgluing}. Up to conjugation there are four $\mathcal{F}$-centric subgroups $Q_1:=\langle x^3,y\rangle$, $Q_2:=\langle x\rangle$, $Q_3:=\langle xy\rangle$ and $D$. This gives seven chains of $\mathcal{F}$-centric subgroups. It can be shown that $\Aut_{\mathcal{F}}(Q_1)\cong S_3$, $\Aut_{\mathcal{F}}(Q_2)\cong C_6$, $\Aut_{\mathcal{F}}(Q_3)\cong C_3$ and $\Aut_{\mathcal{F}}(D)\cong C_3\times S_3$. It follows that $\cohom^2(\Aut_{\mathcal{F}}(\sigma),k^\times)=0$ for all chains $\sigma$ of $\mathcal{F}$-centric subgroups of $D$. Consequently, $\cohom^0([S(\mathcal{F}^c)],\mathcal{A}^2_{\mathcal{F}})=0$. Hence, by Theorem~1.1 in \cite{Parkgluing} the Gluing Problem has at least one solution. (Obviously, this should hold in a more general context.)

Now we determine $\cohom^1([S(\mathcal{F}^c)],\mathcal{A}_{\mathcal{F}}^1)$. For a finite group $A$ it is known that $\cohom^1(A,k^\times)=\Hom(A,k^\times)=\Hom(A/A'{\pcore^{p}}'(A),k^\times)$.
Using this we observe that $\cohom^1(\Aut_{\mathcal{F}}(\sigma),k^\times)\cong C_2$ for all chains except $\sigma=Q_3$ and $\sigma=(Q_3<D)$ in which case we have $\cohom^1(\Aut_{\mathcal{F}}(\sigma),k^\times)=0$.
Since $[S(\mathcal{F}^c)]$ is partially ordered by taking subchains, one can view $[S(\mathcal{F}^c)]$ as a category where the morphisms are given by the pairs of ordered chains. In particular $[S(\mathcal{F}^c)]$ has exactly $13$ morphisms. With the notation of \cite{Webb} the functor $\mathcal{A}_{\mathcal{F}}^1$ is a \emph{representation} of $[S(\mathcal{F}^c)]$ over $\mathbb{Z}$. Hence, we can view $\mathcal{A}_{\mathcal{F}}^1$ as a module $\mathcal{M}$ over the incidence algebra of $[S(\mathcal{F}^c)]$. More precisely, we have
\[\mathcal{M}:=\bigoplus_{a\in\Ob[S(\mathcal{F}^c)]}{\mathcal{A}_{\mathcal{F}}^1(a)}\cong C_2^5.\]
At this point we can apply Lemma~6.2(2) in \cite{Webb}. For this let $d:\Hom[S(\mathcal{F}^c)]\to\mathcal{M}$ a derivation. Then by definition we have $d(\beta)=0$ for $\beta\in\{(Q_3,Q_3),(Q_3,Q_3<D),(D,Q_3<D),(Q_3<D,Q_3<D)\}$. For all identity morphisms $\beta\in\Hom([S(\mathcal{F}^c)])$ we have $d(\beta)=d(\beta\beta)=\mathcal{A}_{\mathcal{F}}^1(\beta)d(\beta)+d(\beta)=2d(\beta)=0$. 
Since $\beta\gamma$ for $\beta,\gamma\in\Hom([S(\mathcal{F}^c)])$ is only defined if $\beta$ or $\gamma$ is an identity, we see that there are no further restrictions on $d$. On the four morphisms $(Q_1,Q_1<D)$, $(D,Q_1<D)$, $(Q_2,Q_2<D)$ and $(D,Q_2<D)$ the value of $d$ is arbitrary. 
It remains to show that $d$ is an inner derivation. For this observe that the map $\mathcal{A}_{\mathcal{F}}^1(\beta)$ is bijective if $\beta$ is one of the four morphisms above. Now we construct a set $u=\{u_a\in\mathcal{A}_{\mathcal{F}}^1(a):a\in\Ob[S(\mathcal{F}^c)]\}$ such that $d$ is the inner derivation induced by $u$. Here we can set $u_{Q_1<D}=0$. Then the equation $d((Q_1,Q_1<D))=\mathcal{A}_{\mathcal{F}}^1((Q_1,Q_1<D))(u_{Q_1})$ determines $u_{Q_1}$. Similarly $d((D,Q_1<D))=\mathcal{A}_{\mathcal{F}}^1(u_D)$ determines $u_D$. Then $d((D,Q_2<D))=\mathcal{A}_{\mathcal{F}}^1(u_D)-u_{Q_2<D}$ gives $u_{Q_2<D}$ and finally $d((Q_2,Q_2<D))=\mathcal{A}_{\mathcal{F}}^1(u_{Q_2})-u_{Q_2<D}$ determines $u_{Q_2}$. 
Hence, Lemma~6.2(2) in \cite{Webb} shows $\cohom^1([S(\mathcal{F}^c)],\mathcal{A}^1_{\mathcal{F}})=0$. So the Gluing Problem has only one solution by Theorem~1.1 in \cite{Parkgluing}.
\end{proof}

Whenever one knows the Cartan matrix (up to basic sets) for a specific defect group, one can apply Theorem~2.4 in \cite{HKS}. This gives the following Corollary.

\begin{Corollary}\label{cartan}
Let $B$ be a $3$-block of a finite group and $(u,b_u)$ be a subsection for $B$ such that $b_u$ has defect group $Q$. If $Q/\langle u\rangle\cong 3^{1+2}_-$, then $k_0(B)\le|Q|$. If in addition $(u,b_u)$ is major, we have $k(B)\le|Q|$, and Brauer's $k(B)$-Conjecture holds for $B$.
\end{Corollary}

Using \cite{Usami23I,UsamiZ2Z2} one can show that Corollary~\ref{cartan} remains true if we replace $3^{1+2}_-$ by the similar group $C_9\times C_3$.

The next interesting case which comes to mind is $p=5$ and $e(B)=4$. Here Proposition~\ref{cor} gives $k(B)\in\{26,27,28\}$, $k_0(B)\in\{22,25\}$, $k_1(B)\in\{1,2,3,4\}$ and $l(B)\in\{4,5,6\}$. It is reasonable that one can settle this and other small cases as well, but this will not necessarily lead to any new insights.

We remark that also for the extraspecial group of order $p^3$ and exponent $p$ some results of Hendren \cite{Hendren1} can be improved. In particular in \cite{HKS} we proved Olsson's Conjecture for these blocks provided $p\ne 3$. On the other hand for $p=3$ Brauer's $k(B)$-Conjecture was shown in \cite{SambalekB2}.

\section*{Acknowledgment}
This work is supported by the German Academic Exchange Service (DAAD). I am grateful to Geoffrey Robinson for showing me \cite{RobObstruction}.

\begin{center}
Benjamin Sambale\\
Mathematisches Institut\\
Friedrich-Schiller-Universität\\
D-07737 Jena\\
Germany\\
\href{mailto:benjamin.sambale@uni-jena.de}{\texttt{benjamin.sambale@uni-jena.de}}
\end{center}

\end{document}